\documentclass[oneside,openany,a4paper]{amsart}
\usepackage[a4paper, margin=1in]{geometry}
\usepackage{amssymb}
\usepackage[centertags]{amsmath}
\usepackage{amsthm}
\usepackage{amsfonts}
\usepackage{eucal}
\usepackage{mathrsfs}
\usepackage{enumerate}

\usepackage[all,cmtip]{xy}
\usepackage{graphicx, eepic}
\usepackage{dsfont, soul, stackrel, tikz-cd, graphicx, etoolbox, multirow}

\usepackage{lmodern}
\usepackage[utf8]{inputenc}
\usepackage[T1]{fontenc}
\usepackage{pdfrender, xcolor}

\usepackage{bbm}
\usepackage{pdfpages}
\usepackage{titletoc}
\usepackage{booktabs}
\usepackage[shortlabels]{enumitem}
\usepackage{mathtools}
\usepackage{thmtools}

\usepackage{epstopdf}
\usepackage{epsfig}
\usepackage{microtype}

\usepackage{hyperref}
\hypersetup{colorlinks}
\hypersetup{
    bookmarks=true,         
    unicode=false,          
    pdftoolbar=true,        
    pdfmenubar=true,        
    pdffitwindow=false,     
    pdfstartview={FitH},    
    pdftitle={My title},    
    pdfauthor={Author},     
    pdfsubject={Subject},   
    pdfcreator={Creator},   
    pdfproducer={Producer}, 
    pdfkeywords={keyword1} {key2} {key3}, 
    pdfnewwindow=true,      
    colorlinks=true,       
    linkcolor=red,          
    citecolor=magenta,        
    filecolor=violet,      
    urlcolor=blue      
}
\usepackage{multirow, url}
\usepackage{textcomp}
\DeclareMathAlphabet{\cmcal}{OMS}{cmsy}{m}{n}
\newtheoremstyle{thm}
  {3pt}
  {3pt}
  {\em}
  {0pt}
  {\bfseries}
  {}
  {5pt}
  {}
\newtheoremstyle{rem}
  {3pt}
  {3pt}
  {}
  {0pt}
  {\bfseries}
  {.}
  {5pt}
  {}

\newtheorem{thm}{Theorem}[section]

\newtheorem{lem}[thm]{Lemma}
\newtheorem{prop}[thm]{Proposition}

\newtheorem{conj}[thm]{Conjecture}

\theoremstyle{definition}
\newtheorem{defn}[thm]{Definition}

\theoremstyle{rem}
\newtheorem{rem}[thm]{{Remark}}

\numberwithin{equation}{section} \numberwithin{table}{section}
	
\newtheorem*{thm*}{Theorem}
\newtheorem*{rem*}{Remark}
\newtheorem*{rems*}{Remarks}
\newtheorem*{exam*}{Example}
\newtheorem*{exams*}{Examples}

\makeatletter
\newcommand{\neutralize}[1]{\expandafter\let\csname c@#1\endcsname\count@}
\makeatother

\usepackage{enumerate}

\def\bos#1{{\mathbf{#1}}}

 \newcommand{\G}{GL_2^{+}(\mathbb{Q})}

 \newcommand{\pa}[1]{\frac{\partial}{\partial{#1}}}
  \newcommand{\prt}[2]{\frac{\partial{#1}}{\partial{#2}}}

  \newcommand{\pr}{\operatorname{prim}}

  \newcommand{\nc}{\newcommand}
  \newcommand{\be}{\begin{eqnarray*}}
  \newcommand{\ee}{\end{eqnarray*}}
  \newcommand{\bea}{\begin{eqnarray}}
  \newcommand{\eea}{\end{eqnarray}}

   \nc{\bei}{\begin{itemize}}
   \nc{\eei}{\end{itemize}}
   \nc{\bee}{\begin{enumerate}}
   \nc{\eee}{\end{enumerate}}
   \nc{\bet}{\begin{thm}}
   \nc{\eet}{\end{thm}}
   \nc{\bed}{\begin{defn}}
   \nc{\eed}{\end{defn}}
   \nc{\bel}{\begin{lem}}
   \nc{\eel}{\end{lem}}
   \nc{\bep}{\begin{prop}}
   \nc{\eep}{\end{prop}}
   \nc{\bec}{\begin{corollary}}
   \nc{\eec}{\end{corollary}}
   \nc{\ber}{\begin{rem}}
   \nc{\eer}{\end{rem}}
   \nc{\beex}{\begin{example}}
   \nc{\eeex}{\end{example}}

   \nc{\bpm}{\begin{pmatrix}}
   \nc{\epm}{\end{pmatrix}}
   \nc{\bspm}{\left(\begin{smallmatrix}}
   \nc{\espm}{\end{smallmatrix}\right)}

\newcommand{\cA}{\mathcal{A}}

\newcommand{\cK}{\mathcal{K}}

\newcommand{\bC}{\mathbb{C}}

\newcommand{\bQ}{\mathbb{Q}}

\newcommand{\bZ}{\mathbb{Z}}

\newcommand{\BP}{\mathbf{P}}

\nc{\frf}{\mathfrak{f}} 

\nc{\frs}{\mathfrak{s}}  
\nc{\frt}{\mathfrak{t}} 
\nc{\fru}{\mathfrak{u}}
  
\nc{\lsl}{\mathfrak{sl}}
\nc{\lgl}{\mathfrak{gl}}

\nc{\upsi}{\underline{\psi}}
\nc{\uchi}{\underline{\chi}}

\DeclareMathOperator{\Spec}{Spec}

\DeclareMathOperator{\Res}{Res}

\newcommand{\lra}{\longrightarrow}    

\nc{\surjto}{\twoheadrightarrow}
\nc{\ts}{\times}
\nc{\ds}{\displaystyle}
\nc{\nd}{\noindent}  
\nc{\ud}{\underline}
\nc{\ov}{\overline}
\nc{\maplra}[1]{\buildrel #1 \over \lra}
\nc{\mapto}[1]{\buildrel #1 \over \to}
\nc{\setb}[1]{\{  #1\}}

 \nc{\cHom}{\mathcal{H}om}

\nc{\cdruur}[8] {\begin{CD} 
#1 @>#2>> #3\\ 
@AA#4A @AA#5A\\ 
#6 @>#7>> #8 
\end{CD} }
\nc{\cdrddr}[8] {\begin{CD} 
#1 @>#2>> #3\\ 
@VV#4V @VV#5V\\ 
#6 @>#7>> #8 
\end{CD} }

\nc{\dia}[8]{\xymatrix{ 
&#1 \ar@{-}[ld]_{#2} \ar@{-}[rd]^{#3} \\
#4 \ar@{-}[rd]_{#6} & &#5 \ar@{-}[ld]^{#7}\\ 
&#8} }

\nc{\diam}[9]{\xymatrix{ 
&#1 \ar@{-}[ld]_{#2}  \ar@{-}[d]^{#3} \ar@{-}[rd]^{#4} \\
#5 \ar@{-}[rd]_{#8}     & #6 \ar@{-}[d]_{#9}      & #7   \ar@{-}[ld]^{2} \\
& \bQ} } 

\nc{\sumn}[2][n]{#2_{1} +#2_{2}+ \cdots + #2_{#1}}
\nc{\poly}[3][n]{#2_{#1}#3^{#1} +#2_{#1-1}#3^{#1-1}  \cdots + #2_{1} #3+ #2_0}
\nc{\dpoly}[3][n]{#1#2_{#1}#3^{#1-1} +(#1-1)#2_{#1-1}#3^{#1-1}  \cdots +2 #2_{2} #3+ #2_1}
\nc{\mpoly}[3][n]{#3^{#1} +#2_{#1-1}#3^{#1-1}  \cdots + #2_{1} #3+ #2_0}

\nc{\vpar}[4]{    \left \{ \begin{array}{cc} #1 & \textrm{if } #2, \\
&\\
#3 & \textrm{if } #4. 
\end{array}\right. }

\nc{\ary}[5]{#1: \left\{ \begin{array}{ll} #2 &\mapsto #3 \\ #4 &\mapsto #5 \end{array} \right.}
 \nc{\bedm}{\begin{displaymath}}
 \nc{\eedm}{\end{displaymath}}
 \nc{\art}{\hbox{\bf Art}^\Z}
 \nc{\bvx}{\bos{B\!\!V}_{\! \!\ud G}}

\newcommand{\pmat}{\left(\begin{matrix}}   
\newcommand{\epmat}{\end{matrix}\right)}   
\newcommand{\psmat}{\left(\begin{smallmatrix}}    
\newcommand{\epsmat}{\end{smallmatrix}\right)}
\nc{\twotwo}[4]{\pmat #1 & #2 \\ #3 & #4 \epmat}
\nc{\thrthr}[9]{\pmat #1 & #2 & #3 \\ #4 & #5 & #6 \\ #7 & #8 & #9 \epmat}
\nc{\stwotwo}[4]{\psmat #1 & #2 \\ #3 & #4 \epsmat}
\nc{\sthrthr}[9]{\psmat #1 & #2 & #3 \\ #4 & #5 & #6 \\ #7 & #8 & #9 \epsmat}

\def\eqalign#1{\null\,\vcenter{\openup\jot\m@th
\ialign{\strut\hfil$\displaystyle{##}$&$\displaystyle{{}##}$\hfil
\crcr#1\crcr}}\,}

\def\eqn#1#2{
\xdef #1{(\nsecsym\the\meqno)}
\global\advance\meqno by1
$$#2\eqno#1\eqlabeL#1
$$}

\def\a{\alpha}
\def\b{\beta}

  \def\G{\Gamma}

\def\l{\lambda}  
\def\m{\mu}

\def\C{\mathbb{C}}

\def\Z{\mathbb{Z}}

\def\mod{\hbox{ }mod\hbox{ }}

\def\wt{\hbox{\it wt}}
\def\ch{\hbox{\it ch}}

\catcode`\@=11 

\global\newcount\nsecno \global\nsecno=0
\global\newcount\meqno \global\meqno=1
\def\newsec#1{\global\advance\nsecno by1
\eqnres@t
\section{#1}}
\def\eqnres@t{\xdef\nsecsym{\the\nsecno.}\global\meqno=1}
\def\sequentialequations{\def\eqnres@t{\bigbreak}}\xdef\nsecsym{}

\def\draftmode{\message{ DRAFTMODE }

{\count255=\time\divide\count255 by 60 \xdef\hourmin{\number\count255}
\multiply\count255 by-60\advance\count255 by\time
\xdef\hourmin{\hourmin:\ifnum\count255<10 0\fi\the\count255}}}
\def\nolabels{\def\wrlabeL##1{}\def\eqlabeL##1{}\def\reflabeL##1{}}
\def\writelabels{\def\wrlabeL##1{\leavevmode\vadjust{\rlap{\smash%
{\line{{\escapechar=` \hfill\rlap{\tt\hskip.03in\string##1}}}}}}}%
\def\eqlabeL##1{{\escapechar-1\rlap{\tt\hskip.05in\string##1}}}%
\def\reflabeL##1{\noexpand\llap{\noexpand\sevenrm\string\string\string##1}
}}

\nolabels

\def\eqn#1#2{
\xdef #1{(\nsecsym\the\meqno)}
\global\advance\meqno by1
$$#2\eqno#1\eqlabeL#1
$$}

\def\eqalign#1{\null\,\vcenter{\openup\jot\m@th
\ialign{\strut\hfil$\displaystyle{##}$&$\displaystyle{{}##}$\hfil
\crcr#1\crcr}}\,}

\vspace{.2in}

   \nc{\hr}{[\![\hbar]\!]}

\catcode`\@=12 

\nc{\bt}{\mathbf{t}}
\draftmode

\begin{document}

\title[Cohomology of complete intersections over a characteristic zero field]{A basis of algebraic de Rham cohomology of complete intersections over a characteristic zero field}%

\author{Jeehoon Park}
\email{jeehoonpark@postech.ac.kr}
\address{Department of Mathematics, POSTECH (Pohang University of Science and Technology), San 31, Hyoja-Dong, Nam-Gu, Pohang, Gyeongbuk, 790-784, South Korea. }
\author{Junyeong Park}
\email{junyeongp@gmail.com}
\address{Department of Mathematics, POSTECH (Pohang University of Science and Technology), San 31, Hyoja-Dong, Nam-Gu, Pohang, Gyeongbuk, 790-784, South Korea. }


\begin{abstract}
Let $\Bbbk$ be a field of characteristic 0.
Let $X$ be a smooth complete intersection over $\Bbbk$ of dimension $n-k$ in the projective space $\BP^n_{\Bbbk}$, for given positive integers $n$ and $k$.
When $\Bbbk=\bC$, Terasoma (\cite{Ter90}) and Konno (\cite{Ko91}) provided an explicit representative (in terms of differential forms) of a basis for the primitive middle-dimensional algebraic de Rham cohomology $H_{dR,\pr}^{n-k}(X;\bC)$.
Later Dimca constructed another explicit representative of a basis of $H_{dR,\pr}^{n-k}(X;\bC)$ in \cite{Dim95}. 
Moreover, he proved that his representative gives the same cohomology class as the previous representative of Terasoma and Konno. The goal of this article is to examine the above two different approaches without assuming that $\Bbbk=\bC$ and provide a similar comparison result for any field $\Bbbk$. Dimca's argument depends heavily on the condition $\Bbbk=\bC$ and our idea is to find appropriate Cech-de Rham complexes and spectral sequences corresponding to those two approaches, which work without restrictions on $\Bbbk$.
\end{abstract}

\maketitle
\tableofcontents


\section{Introduction} \label{sec1}

Let $\Bbbk$ be a field of characteristic zero.
Let $n$ and $k$ be positive integers such that $n \geq k \geq 1$. 
Let $X_{\ud G}$ be a smooth complete intersection variety over $\bQ$ of dimension $n-k$ embedded in the projective space $\BP^{n}$ over $\Bbbk$.
We use $\underline x = [x_0: x_1: \cdots: x_n]$ as a projective coordinate of the projective $n$-space $\BP^n$ and let $G_1(\ud x), \cdots, G_k(\ud x)$ be defining homogeneous polynomials in $\Bbbk[\ud x]$ such that $\deg(G_i)=d_i$ for $i = 1, \cdots, k$.

The main object of study is the primitive middle-dimensional algebraic de Rham cohomology group $H_{dR,\pr}^{n-k}(X_{\ud G};\Bbbk)$.
For this we introduce new variables $y_1, \cdots, y_k$ corresponding to $G_1, \cdots, G_k$. Let $N=n+k+1$ and 
\begin{align}\label{fao}
A :=\Bbbk[y_1, \cdots, y_k, x_0, \cdots, x_n]= \Bbbk[q_\m]_{\m=1,\cdots, N}
\end{align}
where $q_1=y_1, \cdots, q_k=y_k$ and $q_{k+1}=x_0, \cdots, q_N=x_{n}$. Then consider the Dwork potential
\begin{align}\label{dpot}
S(\ud q) := \sum_{\ell=1}^k y_\ell \cdot G_{\ell}(\underline{x}).
\end{align}
When $\Bbbk=\bC$, Konno \cite{Ko91} described $H_{dR,\pr}^{n-k}(X_{\ud G};\bC)$ in terms of Jacobian ideal of $S$: he constructed an isomorphism\footnote{More precisely, the isomorphism is due to Griffiths in the hypersurface case, \cite{Gr69},  Terasoma in the equal degree complete intersection case, \cite{Ter90}, and Konno in the general case, \cite{Ko91}.}
$$
A_{c_{\ud G}}/A_{c_{\ud G}} \cap Jac(S) \simeq H_{dR,\pr}^{n-k}(X_{\ud G};\bC)
$$
where we refer to \eqref{gm} for detailed notations.

When $\Bbbk=\bC$, Dimca found another isomorphism (see subsection \ref{sec2.3} for our review) from $\left({A}/ Jac(S)\right)_{c_{\ud G}}$ to $H_{dR,\pr}^{n-k}(X_{\ud G};\bC)$ in \cite{Dim95}, which is more close to the spirit of Griffiths \cite{Gr69} than Konno by showing that his map (described in a different way) is, in fact, same as the (Terasoma's and) Konno's map via a new use of relative Bochner-Martinelli and Andreotti-Norguet integral formulas for the residue map $\Res_{\ud G}:H^{n+k-1}_{dR}(\BP^n \setminus X_{\ud G}; \bC) \xrightarrow{\sim} H^{n-k}_{dR, \operatorname{prim}}(X_{\ud G};\bC)$. The goal here is to prove a version of Dimca's result when $\bC$ is replaced by any field $\Bbbk$ of characteristic zero.

More precisely, Dimca's result \cite[Proposition 10]{Dim95} (see also Proposition \ref{dc}) can be stated as the commutativity of the (diamond) diagram below when $\Bbbk=\bC$:
$$
\begin{tikzcd}[column sep=tiny]
& H^n_{dR}(\BP^n \setminus D_{\ud G};\Bbbk) \ar[dr, "\delta_{\ud G}"] 
&
&[1.5em] \\
A_{c_{\ud G}}/A_{c_{\ud G}} \cap Jac(S)  \ar[ur, "d_{\ud G}"] \ar[dr, "k_{\ud G}"']
&
& H^{n+k-1}_{dR}(\BP^n\setminus X_{\ud G};\Bbbk) \ar[r,  "\Res_{\ud G}"]
& H^{n-k}_{dR,\pr}(X_{\ud G}; \Bbbk)\\
& H^{n+k-1}_{dR}(\BP(\CMcal{E}) \setminus X_S; \Bbbk) \ar[ur, "s^* "'] 
&
&
\end{tikzcd}
$$
where $s$ is any section to the natural projection map $pr_1:\BP(\CMcal{E}) \setminus X_S \to \BP^n\setminus X_{\ud G}$ and we refer to Definitions  \ref{kgk}, \ref{dgk} and \ref{delgk} for Terasoma-Konno's map $k_{\ud G},$ Dimca's map $d_{\ud G}$ and the natural epimorphism $\delta_{\ud G}$. Note that $k_{\ud G}$ and $s^*$ are isomorphisms but neither $d_{\ud G}$ nor $\delta_{\ud G}$ is an isomorphism unless $k =1$.
Our main theorem \ref{mthm} states that the above diagram commutes (up to a precise numerical factor). Dimca's argument depends heavily on the condition $\Bbbk=\bC$, since Dimca's proof relies on an explicit description (integral formulas) of $\Res_{\ud G}$. On the other hand, the residue map $\Res_{\ud G}: H^{n+k-1}_{dR}(\BP^n\setminus X_{\ud G};\Bbbk) \to H^{n-k}_{dR,\pr}(X_{\ud G}; \Bbbk)$ over a general $\Bbbk$ is defined to be the connecting homomorphism of the distinguished triangle \eqref{XG-dist-triangle} and we do not have an integral formula. So we need a new idea to prove a comparison result for a general field $\Bbbk$: we will find appropriate Cech-de Rham complexes and spectral sequences, which work without restrictions on $\Bbbk$, corresponding to Dimca's and Terasoma-Konno's maps.
 Our argument does not use $\Res_{\ud G}$: it involves direct computations in the level of Cech-de Rham cochains and has more algebraic nature using ``partition of unity''.

The paper consists of two sections. Section \ref{sec2} is a summary of Dimca's result over $\bC$ in terms of the language which is suitable for generalization to arbitrary field of characteristic zero. We set up the basic notations in subsection \ref{sec2.1}. In subsection \ref{sec2.2}, we review Terasoma-Konno's approach to cohomology of complete intersections. In subsection \ref{sec2.3}, we explain Dimca's approach and his comparison result.

Section \ref{sec3} is the main part of the paper. 
In subsection \ref{sec3.1}, we examine the Gysin exact sequence for cohomologies with coefficient $\Bbbk$. Subsection \ref{sec3.2} (respectively, \ref{sec3.3}) is devoted to study of Terasoma-Konno's approach (respectively, Dimca's approaoch) over $\Bbbk$. Finally, in subsection \ref{sec3.4}, we prove the main comparison theorem.

\subsection{Acknowledgement}
The work of Jeehoon Park was partially supported by BRL (Basic Research Lab) through the NRF (National Research Foundation) of South Korea (NRF-2018R1A4A1023590). 
Jeehoon Park thanks KIAS (Korea Institute for Advanced Study), where the part of work was done, for its hospitality.

\section{Dimca's comparison result over $\bC$} \label{sec2}

\subsection{A brief review over complex numbers $\bC$} \label{sec2.1}
We review Terasoma's and Konno's approaches to explicit descriptions of $H^{n-k}_{dR, \pr}(X_{\ud G};\bC)$. Their ideas are same, which are based on the Gysin exact sequence and ``the Cayley trick''. Terasoma's result covers the case when $d_1=\cdots=d_k$ and Konno's result works generally.

Let us briefly explain the Gysin exact sequence and ``the Cayley trick''. There is a long exact sequence, called the Gysin exact sequence:
{
\small{
\begin{eqnarray}\label{gy}
\cdots \to H^{n+k-1}_{dR}(\BP^n ; \bC) \to H^{n+k-1}_{dR}(\BP^n \setminus X_{\ud G};\bC) \xrightarrow{\Res_{\ud G}} H^{n-k}_{dR}(X_{\ud G};\bC) \xrightarrow{\operatorname{Gys}} H^{n+k}_{dR}(\BP^n;\bC)  \to \cdots
\end{eqnarray}
}}
where $\Res_{\ud G}$ is the residue map (see p.\,96 of \cite{Dim95}) and $\operatorname{Gys}$ is the Gysin map (the cup product with $k$-th wedge product of the fundamental K\"{a}hler 2-form on $X_{\ud G}$). This sequence gives rise to an isomorphism 
\begin{eqnarray}\label{residuemap}
\Res_{\ud G}:H^{n+k-1}_{dR}(\BP^n \setminus X_{\ud G}; \bC) \xrightarrow{\sim} H^{n-k}_{dR, \operatorname{prim}}(X_{\ud G};\bC),
\end{eqnarray}
where the primitive cohomology $H^{n-k}_{dR, \operatorname{prim}}(X_{\ud G};\bC)$ is defined to be the kernel of $\operatorname{Gys}$.
The Cayley trick is about translating a computation of the cohomology of the complement of a complete intersection into a computation of the cohomology of the complement of a hypersurface in a bigger space. Let $\CMcal{E}=\CMcal{O}_{\BP^n}(d_1) \oplus \cdots \oplus \CMcal{O}_{\BP^n}(d_k)$ be the locally free sheaf
of $\CMcal{O}_{\BP^n}$-modules with rank $k$. Let $\BP(\CMcal{E})$ be the projective bundle associated to $\CMcal{E}$ with fiber $\BP^{k-1}$ over $\BP^n$. Then $\BP(\CMcal{E})$ is the smooth projective toric variety with Picard group isomorphic to $\bZ^2$ whose (toric) homogeneous coordinate ring\footnote{This was already introduced in \eqref{fao}} is given by
\begin{eqnarray} \label{fa}
A:=A_{\BP(\CMcal{E})}:=\bC[y_1, \cdots, y_k, x_0, \cdots, x_n]
\end{eqnarray}
where $y_1, \cdots, y_k$ are new variables corresponding to $G_1, \cdots, G_k$. This is the polynomial ring introduced in \eqref{fao}.
There are two additive gradings $\ch$ and $\wt$, called the charge and the weight, corresponding to the Picard group $\bZ^2$: 
$$
\ch(y_i)=-d_i, \quad \text{for } i=1, \cdots, k, \quad \ch(x_{j}) = 1, \quad \text{for }  j=0, \cdots, n,
$$ 
$$
\wt(y_i)=1, \quad \text{for } i=1, \cdots, k, \quad \wt(x_{j}) = 0, \quad \text{for }  j=0, \cdots, n. 
$$ 
We have the weight and the charge decomposition of $A$ such that
\[
	A =\bigoplus_{\l \in \Z}\bigoplus_{w \geq 0}A_{\l, (w)}
\]
where subindices $\l$ and $(w)$ means the charge and the weight respectively.
Then 
$$
S(\ud y, \ud x):= \sum_{j=1}^k y_j G_j(\ud x) \in {A_{0,(1)}}
$$
defines a hypersurface $X_S$ in $\BP(\CMcal{E})$.
The natural projection map $pr_1:\BP(\CMcal{E}) \to \BP^n$ induces a morphism 
$
pr_1: \BP(\CMcal{E})\setminus X_S \to \BP^n \setminus X_{\ud G}
$
which can be checked to be a homotopy equivalence (the fibers are affine spaces). Hence there exists an isomorphism 
\begin{eqnarray}\label{secs}
H^{n+k-1}_{dR}(\BP(\CMcal{E}) \setminus X_S; \bC) \xrightarrow{s^*} H^{n+k-1}_{dR}(\BP^n \setminus X_{\ud G}; \bC)
\end{eqnarray}
where $s$ is a section to $pr_1$.
The cohomology group $H^{n+k-1}_{dR}(\BP(\CMcal{E}) \setminus X_S; \bC)$ of a hypersurface complement in $\BP(\CMcal{E})$ can be described explicitly in terms of the de Rham cohomology of $\BP(\CMcal{E})$ with poles along $X_S$.
Terasoma and Konno proved\footnote{More precisely, Terasoma (\cite{Ter90}) provided such an isomorphism in the case $d_1=d_2=\cdots=d_k$ and Konno (\cite{Ko91}) extended the result of Terasoma to the general case when $d_j$'s are not equal. Also see \cite{Gr69} for the pioneering work of Griffiths in the case $k=1$, hypersurface case.} that there is an isomorphism
\begin{eqnarray}\label{gm}
 \Phi_S: A_{c_{\ud G}}/A_{c_{\ud G}} \cap Jac(S) \xrightarrow{\sim}  H^{n+k-1}_{dR}(\BP(\CMcal{E}) \setminus X_S; \bC),
\end{eqnarray}
where
$Jac(S)$ is the Jacobian ideal of $S(\ud y, \ud x)$, and 
$$
c_{\ud G} := \sum_{i=1}^k d_i - (n+1).
$$
Here the sub-index $c_{\ud G}$ means the submodule in which the charge is $c_{\ud G}$ and $Jac(S)$ is the sum of the images of the endomorphisms $\prt{ S }{y_i}, \prt{S}{x_j}$ of ${A}$ ($i =1, \cdots, k, j=0, \cdots, n$).


\subsection{An explicit basis due to Terasoma and Konno}\label{sec2.2}
To review the explicit basis of Terasoma and Konno, we need to examine a map from $ \left({A}/ Jac(S)\right)_{c_{\ud G}}$ to $H^{n+k-1}_{dR}(\BP(\CMcal{E}) \setminus X_S; \bC)$.



Let
{\small{
\begin{eqnarray}\label{oxy}
\Omega_x = \sum_{i=0}^n (-1)^i x_i (dx_0\wedge \cdots \wedge \hat {d x_i} \wedge \cdots \wedge dx_n), \quad 
\Omega_y = \sum_{i=1}^n (-1)^i y_i (dy_1\wedge \cdots \wedge \hat {d y_i} \wedge \cdots \wedge dy_k).
\end{eqnarray}
}}
For a multi-index $\ud i=(i_1, \cdots, i_k)$ with $|\ud i|=i_1 + \cdots +i_k$, consider the rational differential form (given in \cite[p 90]{Dim95} and originally studied by \cite{Ter90} and \cite{Ko91})
\begin{eqnarray}\label{alp}
\a(S,\ud y^{\ud i}u(\ud x) ):=(-1)^{n(k-1)+|\ud i|}(k+|\ud i|-1)!\frac{\ud y^{\ud i}u(\ud x) }{S(\ud q)^{k+|\ud i|}} \Omega_x \Omega_y \in \Omega^{N-2}(\mathbf{P}(\mathcal{E})\setminus X_S)
\end{eqnarray}
where $\ud y^{\ud i}= y_1^{i_1} \cdots y_k^{i_k}$ and $\ud y^{\ud i}u(\ud x) \in \cA_{c_{\ud G}}^0$.  
  Define $k_{\ud G}:  A_{c_{\ud G}}/A_{c_{\ud G}} \cap Jac(S) \to H^{n+k-1}_{dR}(\BP(\CMcal{E}) \setminus X_S; \bC)$ as follows:
  \begin{eqnarray} \label{kg}
  k_{\ud G} ([\ud y^{\ud i}u(\ud x)]) = [\a(S,\ud y^{\ud i}u(\ud x) )]
  \end{eqnarray}
  for $\ud y^{\ud i}u(\ud x) \in A_{c_{\ud G}}$ (linearly extended to every element of $A_{c_{\ud G}}$). Here $[\cdot]$ is the cohomology class.
 Konno (\cite{Ko91}) proved that this map $k_{\ud G}$ is an isomorphism of finite dimensional $\bC$-vector spaces.  

 \begin{defn} [Terasoma's and Konno's approaches]
 We call the map $k_{\ud G}$ as the Terasoma/Konno map.
 \end{defn}
 

 \subsection{An explicit basis due to Dimca and a comparison}\label{sec2.3}
 Here we review Dimca's approach and his comparison result with Terasoma's and Konno's approaches. 
 
 We consider the natural epimorphism which appeared in \cite[p 95]{Dim95},
\[
\delta_{\ud G}: H^n_{dR}(\BP^n \setminus D_{\ud G};\bC) \to H^{n+k-1}_{dR}(\BP^n\setminus X_{\ud G};\bC),
\]
where $D_{\ud G}$ is the divisor defined by $G_1(\ud x) \cdots G_k(\ud x) =0$.
 
Dimca considered the following rational differential form (\cite[p 93]{Dim95})
\[
\b(S, \ud y^{\ud i} u(\ud x)):=(-1)^{|\ud i|+1}\frac{i_1! \cdots i_k!\cdot u(\ud x)}{G_1^{i_1+1}(\ud x) \cdots G_k^{i_k+1}(\ud x)} \Omega_x \in \Omega^n(\BP^n \setminus D_{\ud G}),
\]
for $\ud y^{\ud i} u(\ud x) \in A_{c_{\ud G}}$. 
This explicit representation has a benefit of being written as the Laplace transform\footnote{This fact is crucially used to develop the deformation theory of period integrals of $X_{\ud G}$ in \cite{KPP20}.}
\begin{eqnarray*}\label{laplace}
 -\int_{0}^\infty \cdots \int_{0}^\infty  \ud y^{\ud i} u(\ud x) e^{S(\ud q)} dy_1 \cdots dy_k  
= \frac{(-1)^{|\ud i|+1} i_1 ! \cdots i_k !  u(\ud x)  }{ G_1^{i_1+1} \cdots G_k^{i_k+1}}
\end{eqnarray*}
 for $\ud y^{\ud i} u(\ud x) \in A_{c_{\ud G}}$. 
 Define $d_{\ud G}:  A_{c_{\ud G}}/A_{c_{\ud G}} \cap Jac(S) \to H^n_{dR}(\BP^n \setminus D_{\ud G};\bC) $ as follows:
  \begin{eqnarray} \label{dg}
  d_{\ud G} ([\ud y^{\ud i}u(\ud x)]) = [\b(S,\ud y^{\ud i}u(\ud x) )], \quad \ud y^{\ud i} u(\ud x) \in A_{c_{\ud G}}.
  \end{eqnarray}
 
 \begin{defn}[Dimca's approach]
 We call the map $d_{\ud G}$ as the Dimca map.
 \end{defn}

 In order to describe Dimca's comparison result, we need to know an explicit description of a section $s$ of the second map $s^*$ in \eqref{secs}. For this, we review the toric quotient construction of the projective bundle $\BP(\CMcal{E})$. Let $G=\C^*\times \C^*$. Let $U_G=(\C^{n+1}\setminus \{0\}) \times (\C^{k}\setminus \{0\})$, consider the following $G$-action on $U_G$:
\[
(u,v) (\ud x,\ud y)=(u x_0, \cdots, ux_n, u^{-d_1} vy_1, \cdots, u^{-d_k} v y_k).
\]
The hypersurface $X_S$ in $U_G$ is $G$-invariant. 
The lemma 17, \cite{Dim95}, says that the geometric quotient $U_G/G$ is naturally identified to the projective bundle $\BP(\CMcal{E})$.
 Let $d=lcm(d_1, \cdots, d_k)$ and define the positive integers $e_i=d/d_i$. According to page 98, \cite{Dim95}, one has a well-defined section
\begin{eqnarray} \label{section}
s_{\ud G}(\ud x) =(\ud x, G_1^{e_1-1}(\ud x)\overline G_1^{e_1}(\ud x), \cdots, G_1^{e_k-1}(\ud x)\overline G_1^{e_k}(\ud x) )
\end{eqnarray}
to the projection map $pr_1:\BP(\CMcal{E})\setminus X_S \to \BP^n \setminus X_{\ud G}$.

The following is the main result of Dimca \cite[Proposition 10]{Dim95} which compares his approach to Terasoma-Konno's approach.
\begin{prop}[the proposition 10, \cite{Dim95}] \label{dc}
We have the following equality
\[
\delta_{\ud G} ([\b(S, \ud y^{\ud i} u (\ud x)])=s_{\ud G}^* ([\a(S, \ud y^{\ud i} u(\ud x)])
\]
for $[\ud y^{\ud i} u (\ud x)] \in A_{c_{\ud G}}/A_{c_{\ud G}} \cap Jac(S)$. In other words, we have
\[
\delta_{\ud G} \circ d_{\ud G} =  s_{\ud G}^* \circ k_{\ud G}
\quad (\text{or}\quad pr_1^\ast \circ \delta_{\ud G} \circ d_{\ud G} =  k_{\ud G} ).
\]
\end{prop}

%

\section{Dimca's comparison over any field $\Bbbk$ of characteristic 0}\label{sec3}

In this section, we will work with a field $\Bbbk$ of characteristic $0$ and homogeneous polynomials $G_1(\underline{x}),\cdots,G_k(\underline{x})\in\Bbbk[\underline{x}]$ defining a smooth complete intersection $X_{\underline{G}}$ in $\mathbf{P}^n$ over $\Bbbk$. 
In order to simplify the notations, we will use the notations as the $\bC$-case for $A, \Res_{\ud G}, k_{\ud G}, d_{\ud G},$ and $\delta_{\ud G}$.\footnote{For example, $A=\Bbbk[\ud y, \ud x]$, $\Res_{\ud G}: H^{q}_{dR}(\mathbf{P}^n\setminus X_{\ud G};\Bbbk) \to H^{q-2k+1}_{dR}(X_{\underline{G}};\Bbbk).$}

Throughout this section, we will use the affine open covering
\begin{align}\label{XG-complement-covering}
\mathfrak{U}:=\{D_+(G_\lambda)\}_{\lambda=1,\cdots,k}
\end{align}
of $\mathbf{P}^n\setminus X_{\underline{G}}$ where $D_+(G_\lambda)$ denotes the nonvanishing locus of $G_\lambda$ in $\mathbf{P}^n$ (see \cite[Proposition 2.5]{Har77} for the notation). Denote $D_{\underline{G}}\subseteq\mathbf{P}^n$ the divisor cut out by $G_1\cdots G_k$. Then
\begin{align}\label{DG-complement}
D_+(G_1)\cap\cdots\cap D_+(G_k)=D_+(G_1\cdots G_k)=\mathbf{P}^n\setminus D_{\underline{G}}.
\end{align}

\subsection{The Gysin exact sequence} \label{sec3.1}
Here we show that the Gysin exact sequence \eqref{gy} holds when $\bC$ is replaced by $\Bbbk$.

Given an $\CMcal{O}_{\mathbf{P}^n}$-module $\CMcal{F}$, denote $\underline{H}_{X_{\underline{G}}}(\CMcal{F})\subseteq\CMcal{F}$ the subsheaf of sections supported on $X_{\underline{G}}$. Since the assignment $\CMcal{F}\mapsto\underline{H}_{X_{\underline{G}}}(\CMcal{F})$ defines a left exact functor, there are derived functors $\underline{H}_{X_{\underline{G}}}^p:=\mathbf{R}^p\underline{H}_{X_{\underline{G}}}$. Since $X_{\underline{G}}\subseteq\mathbf{P}^n$ is a smooth complete intersection, we can describe $\underline{H}_{X_{\underline{G}}}^p(\CMcal{F})$ using the affine open covering $\mathfrak{U}=\{D_+(G_\lambda)\}_{\lambda=1,\cdots,k}$ of (\ref{XG-complement-covering}). Let 
\begin{align*}
\CMcal{F}_{\lambda_0\cdots\lambda_p}:=\left(D_+(G_{\lambda_0}\cdots G_{\lambda_p})\hookrightarrow\mathbf{P}^n\right)_\ast\left(\CMcal{F}|_{D_+(G_{\lambda_0}\cdots G_{\lambda_p})}\right)
\end{align*}
where $\{\l_0, \cdots, \l_p \} \subset \{1, \cdots, k \}$.
Then there is the associated \v{C}ech complex of $\CMcal{O}_{\mathbf{P}^n}$-modules:
\begin{align*}
\xymatrix{\check{\mathfrak{C}}^\bullet(\mathfrak{U},\CMcal{F}):0 \ar[r] & \CMcal{F} \ar[r] & \displaystyle\prod_{\lambda_0}\CMcal{F}_{\lambda_0} \ar[r] & \displaystyle\prod_{\lambda_0<\lambda_1}\CMcal{F}_{\lambda_0\lambda_1} \ar[r] & \cdots \ar[r] & \CMcal{F}_{1\cdots k} \ar[r] & 0}
\end{align*}
where $\CMcal{F}$ is put in cohomological degree $0$.
\begin{lem}\label{XG-supported-coh} If $\CMcal{F}$ is a quasicoherent $\CMcal{O}_{\mathbf{P}^n}$-module, then the following hold.
\begin{quote}
(1) $\check{\mathfrak{C}}^\bullet(\mathfrak{U},\CMcal{F})$ represents $\mathbf{R}\underline{H}_{X_{\underline{G}}}(\CMcal{F})$ in the derived category of $\CMcal{O}_{\mathbf{P}^n}$-modules whose cohomology sheaves are supported on $X_{\underline{G}}$. In other words,
\begin{align*}
\underline{H}_{X_{\underline{G}}}^p(\CMcal{F})\cong H^p(\check{\mathfrak{C}}^\bullet(\mathfrak{U},\CMcal{F}))
\end{align*}
holds for every $p\in\mathbb{Z}$.\\
(2) There is a canonical map
\begin{align*}
\xymatrix{\CMcal{F}\otimes_{\CMcal{O}_{\mathbf{P}^n}}\CMcal{O}_{X_{\underline{G}}} \ar[r] & \underline{H}_{X_{\underline{G}}}^k(\CMcal{F}) & s \ar@{|->}[r] & \displaystyle\frac{s}{G_1\cdots G_k}}
\end{align*}
(3) If, furthermore, $\CMcal{F}$ is flat over $\CMcal{O}_{\mathbf{P}^n}$, then $H_{X_{\underline{G}}}^p(\CMcal{F})=0$ for $p\neq k$.
\end{quote}
\end{lem}
\begin{proof}
For (1), see \cite[\href{https://stacks.math.columbia.edu/tag/0G7M}{Tag 0G7M}]{Stacks}. For (2), see \cite[\href{https://stacks.math.columbia.edu/tag/0G7Q}{Tag 0G7Q}]{Stacks}. For (3), see \cite[\href{https://stacks.math.columbia.edu/tag/0G7P}{Tag 0G7P}]{Stacks}.
\end{proof}
On the other hand, for the closed embedding $i:X_{\underline{G}}\hookrightarrow\mathbf{P}^n$ and the open inclusion $j:(\mathbf{P}^n\setminus X_{\underline{G}})\hookrightarrow\mathbf{P}^n$, there is a distinguished triangle (see \cite[\href{https://stacks.math.columbia.edu/tag/0G72}{Tag 0G72}]{Stacks})
\begin{align}\label{XG-dist-triangle}
\xymatrix{i_\ast\mathbf{R}\underline{H}_{X_{\underline{G}}}(\Omega_{\mathbf{P}^n/\Bbbk}^\bullet) \ar[r] & \Omega_{\mathbf{P}^n/\Bbbk}^\bullet \ar[r] & \mathbf{R}j_\ast\Omega_{(\mathbf{P}^n\setminus X_{\underline{G}})/\Bbbk}^\bullet \ar[r] & i_\ast\mathbf{R}\underline{H}_{X_{\underline{G}}}(\Omega_{\mathbf{P}^n/\Bbbk}^\bullet)[1]}.
\end{align}
Using Lemma \ref{XG-supported-coh}, we may choose a representative of $\mathbf{R}\underline{H}_{X_{\underline{G}}}(\Omega_{\mathbf{P}^n/\Bbbk}^\bullet)$. Using Cartan-Eilenberg resolutions, we get a spectral sequence (see \cite[\href{https://stacks.math.columbia.edu/tag/015J}{Tag 015J}]{Stacks}):
\begin{align*}
E_1^{p,q}=\underline{H}_{X_{\underline{G}}}^p(\Omega_{\mathbf{P}^n/\Bbbk}^q)\implies\underline{H}_{X_{\underline{G}}}^{p+q}(\Omega_{\mathbf{P}^n/\Bbbk}^\bullet)
\end{align*}
This spectral sequence degenerates by Lemma \ref{XG-supported-coh}:
\begin{align*}
E_1^{p,q}=\left\{\begin{array}{ll}
\underline{H}_{X_{\underline{G}}}^k(\Omega_{\mathbf{P}^n/\Bbbk}^q) & \textrm{if}\quad p=k \\
0 & \textrm{otherwise}
\end{array}\right.
\end{align*}
which implies
\begin{align*}
\mathbf{R}\underline{H}_{X_{\underline{G}}}(\Omega_{\mathbf{P}^n/\Bbbk}^\bullet)\cong\underline{H}_{X_{\underline{G}}}^k(\Omega_{\mathbf{P}^n/\Bbbk}^\bullet)[-k]
\end{align*}
To compute the right hand side, consider the conormal exact sequence (see \cite[Proposition 8.12]{Har77}):
\begin{align*}
\xymatrix{0 \ar[r] & \CMcal{C}_{X_{\underline{G}}/\mathbf{P}^n} \ar[r] & \Omega_{\mathbf{P}^n/\Bbbk}^1\otimes_{\CMcal{O}_{\mathbf{P}^n}}\CMcal{O}_{X_{\underline{G}}} \ar[r] & \Omega_{X_{\underline{G}}/\Bbbk}^1 \ar[r] & 0}
\end{align*}
where $\CMcal{C}_{X_{\underline{G}}/\mathbf{P}^n}$ is the conormal module of the embedding $X_{\underline{G}}\hookrightarrow\mathbf{P}^n$; the exactness at $\CMcal{C}_{X_{\underline{G}}/\mathbf{P}}$ follows from the smoothness of $X_{\underline{G}}$ over $\Bbbk$. Hence the composition 
\begin{align*}
\xymatrix{\Omega_{\mathbf{P}^n/\Bbbk}^\bullet \ar[r] & \Omega_{\mathbf{P}^n/\Bbbk}^\bullet\otimes_{\CMcal{O}_{\mathbf{P}^n}}\CMcal{O}_{X_{\underline{G}}} \ar[r] & \Omega_{X_{\underline{G}}/\Bbbk}^\bullet}
\end{align*}
is surjective (the first map is obviously surjective). In this setting, we may define the \emph{Gysin map} as follows:
\begin{align*}
\xymatrix{\gamma_{\underline{G}}^q:\Omega_{X_{\underline{G}}/\Bbbk}^q \ar[r] & \underline{H}_{X_{\underline{G}}}^k(\Omega_{\mathbf{P}^n/\Bbbk}^{q+k}) & \omega \ar@{|->}[r] & \displaystyle\frac{\widetilde{\omega}|_{X_{\underline{G}}}}{G_1\cdots G_k}\wedge dG_1\wedge\cdots\wedge dG_k}.
\end{align*}
Here we may take a lift $\widetilde{\omega}$ for a given local section $\omega$ over affine open subsets, since $\Omega_{\mathbf{P}^n/\Bbbk}^q\rightarrow\Omega_{X_{\underline{G}}/\Bbbk}^q$ still remains surjective over affine open subsets. Since any two lifts of $\omega$ differ by sections of the form
\begin{align*}
\sum_{\lambda=1}^kG_\lambda\alpha_\lambda+\sum_{\lambda=1}^kdG_\lambda\wedge\beta_\lambda
\end{align*}
which maps to zero by construction, $\gamma_{\underline{G}}^q$ are well-defined. Moreover, $\gamma_{\underline{G}}^q$ gather to define a cochain map (see \cite[\href{https://stacks.math.columbia.edu/tag/0G87}{Tag 0G87}]{Stacks}) so we get a map
\begin{align*}
\xymatrix{\gamma_{\underline{G}}:\Omega_{X_{\underline{G}}/\Bbbk}^\bullet \ar[r] & \underline{H}_{X_{\underline{G}}}^k(\Omega_{\mathbf{P}^n/\Bbbk}^\bullet)[k] \ar[r]^-\sim & \mathbf{R}\underline{H}_{X_{\underline{G}}}(\Omega_{\mathbf{P}^n/\Bbbk}^\bullet)[2k]}
\end{align*}
This map is a quasi-isomorphism (see \cite[II, Lemma (3.1)]{Har75} and its proof). Therefore, the distinguished triangle (\ref{XG-dist-triangle}) is isomorphic to
\begin{align*}
\xymatrix{i_\ast\Omega_{X_{\underline{G}}}^\bullet[-2k] \ar[r] & \Omega_{\mathbf{P}^n/\Bbbk}^\bullet \ar[r] & \mathbb{R}j_\ast\Omega_{(\mathbf{P}^n\setminus X_{\underline{G}})/\Bbbk}^\bullet \ar[r] & i_\ast\Omega_{X_{\underline{G}}}^\bullet[-2k+1]}
\end{align*}
By taking the derived global section $\mathbf{R}\Gamma(\mathbf{P}^n,-)$, the associated long exact sequence gives us
\begin{align}\label{gysin}
\xymatrixcolsep{1.5pc}\xymatrix{\cdots \ar[r] & H^q_{dR}(\mathbf{P}^n;\Bbbk) \ar[r] & H^q_{dR}(\mathbf{P}^n\setminus X_{\underline{G}};\Bbbk) \ar[r] & H^{q-2k+1}_{dR}(X_{\underline{G}};\Bbbk) \ar[r] & H^{q+1}_{dR}(\mathbf{P}^n;\Bbbk) \ar[r] & \cdots}
\end{align}

\begin{defn}
We define the residue map $\Res_{\underline{G}}:H^q_{dR}(\mathbf{P}^n\setminus X_{\underline{G}};\Bbbk) \to H^{q-2k+1}_{dR}(X_{\underline{G}};\Bbbk)$ to be the connecting homomorphism in \eqref{gysin}. 
\end{defn}

\subsection{Terasoma-Konno's approach over $\Bbbk$}\label{sec3.2}


\begin{defn}[Terasoma-Konno's map over $\Bbbk$] \label{kgk}
We define $k_{\ud G}$ exactly same as \eqref{kg}, which works for any field of characteristic 0.
\end{defn}

Then the map $k_{\ud G}$ is an $\Bbbk$-linear isomorphism:
\begin{eqnarray}\label{kkg}
k_{\ud G}:  A_{c_{\ud G}}/A_{c_{\ud G}} \cap Jac(S) \xrightarrow{\sim} H^{n+k-1}_{dR}(\mathbf{P}(\mathcal{E})\setminus X_S;\Bbbk).
\end{eqnarray}
This is because one can find a $\bQ$-basis $\mu$ of $\bQ[\ud q]_{c_{\ud G}}/\bQ[\ud q]_{c_{\ud G}}\cap Jac(S)$ and a $\bQ$-basis $\nu$ of $H^{n+k-1}_{dR}(\mathbf{P}(\mathcal{E})\setminus X_S;\bQ)$ such that $\mu$ (respectively, $\nu$) extends to a $\bC$-basis of  $\bC[\ud q]_{c_{\ud G}}/\bC[\ud q]_{c_{\ud G}}\cap Jac(S)$ (respectively, $H^{n+k-1}_{dR}(\mathbf{P}(\mathcal{E})\setminus X_S;\bC)$).   
Since $\Bbbk$ is a $\bQ$-algebra, we can extend theses bases to $\Bbbk$-bases to obtain the isomorphism in \eqref{kkg}. 


Here we provide a detailed structure of the map $k_{\ud G}$.
For any commutative $\Bbbk$-algebra C, let us consider the de Rham complex $(\Omega^\bullet_C, d)$ and the twisted de Rham complex $(\Omega^\bullet_C, d+ dS \wedge)$ for any element $S\in C$. We define the charge and the weight on the de Rham complex by
\[
	\begin{cases}
		\ch(q_i)=\ch(dq_i), &i=1,2,\dots,N\\
		\wt(q_i)=\wt(dq_i). &i=1,2,\dots,N\\
	\end{cases}
\]
We have the weight and the charge decomposition of $\Omega$ such that
\[
	\Omega = \bigoplus_{-N\leq j \leq 0}\bigoplus_{\l \in \Z}\bigoplus_{w \geq 0}\Omega^j_{\l, (w)}.
\]

Let $B:=\Bbbk[\underline{q},S^{-1}]_{0,(0)}$. Then $\mathbf{P}(\mathcal{E})\setminus X_S$ is a smooth affine variety whose coordinate ring is given by $B$.

\begin{defn} 
We define a sequence of maps as follows (here $N=n+k+1$):

\end{defn}

We define a map $\varphi_S$ as the composition of the following maps:
\[
	\begin{tikzcd}
		A_{c_{\ud G}} \arrow[r,"\mu"] & \left(\Omega^N_{A}\right)_0 \arrow[r,"\rho"] & \left(\Omega^N_{A[S^{-1}]}\right)_{0,(0)} \arrow[r,"\theta_{wt}\circ\theta_{ch}"] & \Omega_{B}^{N-2} \arrow[r,"\textrm{quotient}"] & \Omega_B^{N-2}/d(\Omega_B^{N-3})=H_{\mathrm{dR}}^{N-2}(\mathbf{P}(\mathcal{E})\setminus X_S;\Bbbk)
	\end{tikzcd}
\]
where 
	\begin{enumerate}[(i)]
		\item 
			For $\underline{q}^{\underline{u}}:= q_1^{u_1} \cdots q_N^{u_N} \in A_{c_{\ud G}}$,
	        $$
	        \mu(\underline{q}^{\underline{u}})=-\underline{q}^{\underline{u}}dq_1\wedge \cdots \wedge dq_N. 
	        $$
	
		\item  For $\underline{x}^{\underline{u}} = x_0^{u_0} \cdots x_n^{u_n}$, $\underline{y}^{\underline{v}}:=y_1^{v_1}\cdots y_k^{v_k}$ with  $|\ud v| =v_{1} + \cdots + v_{k}$,
                $$
                \rho(\underline{x}^{\underline{u}}\underline{y}^{\underline{v}} dq_1 \wedge \cdots \wedge dq_N )=(-1)^{|\underline{v}|+k-1}(|\underline{v}|+k-1)!\frac{\underline{x}^{\underline{u}}\underline{y}^{\underline{v}}}{S^{|\underline{v}|+k}} dq_1 \wedge \cdots \wedge dq_N.
                $$

		\item Let $\theta_{\ch}$ be the contraction operator with the vector field $\sum_{i=1}^N\ch(q_i)q_i\frac{\partial}{\partial q_i}$, and $\theta_{\wt}$ be the contraction operator with the vector field $\sum_{i=1}^N\wt(q_i)q_i\frac{\partial}{\partial q_i}$.
	\end{enumerate}

Consider the rational differential form same as \eqref{alp}
\[
\a(S,\ud y^{\ud i}u(\ud x) ):=(-1)^{n(k-1)+|\ud i|}(k+|\ud i|-1)!\frac{\ud y^{\ud i}u(\ud x) }{S(\ud q)^{k+|\ud i|}} \Omega_x \Omega_y \in \Omega_B^{N-2}
\]
where $\ud y^{\ud i}u(\ud x) \in \cA_{c_{\ud G}}^0$. Then a simple computation confirms that
\begin{eqnarray}\label{result}
\varphi_S (\ud y^{\ud i}u(\ud x) )= [\a(S,\ud y^{\ud i}u(\ud x) )]
\end{eqnarray}
where $[\cdot]$ means the cohomology class.

\begin{prop}\label{kerp}
The kernel of the map $\varphi_S$ is $\cK_{\ud G} \cap A_{c_{\ud G}}$ where
\begin{eqnarray}\label{kker}
\cK_{\ud G}:=\bigoplus_{i=1}^N \left( \pa{q_i}+\prt{ S(\ud q) }  {q_i}\right) A. 
\end{eqnarray}
\end{prop}
 By this proposition, $\varphi_S$ induces a map $\bar \varphi_S: A_{c_{\ud G}}/A_{c_{\ud G}} \cap \cK_{\ud G} \to H^{n+k-1}_{dR}(\BP(\CMcal{E}) \setminus X_S; \Bbbk) $.
 \begin{prop}
 Let $\{u_\a + Jac(S)\}$ be a $\Bbbk$-basis of $A_{c_{\ud G}}/A_{c_{\ud G}} \cap Jac(S)$. The assignment $u_\a + Jac(S) \mapsto u_\a+ \cK_{\ud G}$ provides us an isomorphism, denoted by $q_S$
 \begin{eqnarray}
 q_S:A_{c_{\ud G}}/A_{c_{\ud G}} \cap Jac(S)\xrightarrow{\sim} A_{c_{\ud G}}/A_{c_{\ud G}} \cap \cK_{\ud G}.
 \end{eqnarray}
  \end{prop}

  We define $J_{\ud G}: {A_{c_{\ud G}}}  \to H^{n-k}_{dR,\pr}(X_{\ud G}; \Bbbk)$ by the composition of the following maps:
 {\small{
 \begin{eqnarray}\label{gma}
\qquad \quad J_{\ud G}: {A_{c_{\ud G}}} \xrightarrow{\varphi_S}   H^{n+k-1}_{dR}(\BP(\CMcal{E}) \setminus X_S; \Bbbk) \xrightarrow{s^*}  H^{n+k-1}_{dR}(\BP^n \setminus X_{\ud G}; \Bbbk)
\xrightarrow{\Res_{\ud G}} H^{n-k}_{dR,\pr}(X_{\ud G}; \Bbbk)
\end{eqnarray}
 }}
 where we choose a section $s$ to the projection $pr_1$.
 Denote the induced maps on $A_{c_{\ud G}}/A_{c_{\ud G}}$ by $\bar \varphi_S$ and $\bar J_{\ud G}$. Therefore we have an isomorphism
 $$
 A_{c_{\ud G}}/A_{c_{\ud G}} \cap Jac(S) \xrightarrow{q_S} A_{c_{\ud G}}/A_{c_{\ud G}} \cap \cK_{\ud G} \xrightarrow{\bar J_{\ud G}=\Res_{\ud G} \circ s^* \circ\bar \varphi_S} H_{dR,\pr}^{n-k}(X_{\ud G};\Bbbk).
 $$

 The Terasoma/Konno map $k_{\ud G}$ over $\Bbbk$ can be understood as $\bar \varphi_{S} \circ q_S:A_{c_{\ud G}}/A_{c_{\ud G}} \cap Jac(S) \to H^{n+k-1}_{dR}(\BP(\CMcal{E}) \setminus X_S; \Bbbk)$.

\subsection{Dimca's map over $\Bbbk$ and construction of $\delta_{\protect\underline{G}}$} \label{sec3.3}

\begin{defn}[Dimca's map over $\Bbbk$] \label{dgk}
We define $d_{\ud G}$ exactly same as \eqref{dg}, which works for any field of characteristic 0.
\end{defn}

 Now we study how to construct $\delta_{\ud G}$ over $\Bbbk$.
The affine open covering $\mathfrak{U}=\{D_+(G_\lambda)\}_{\lambda=1,\cdots,k}$ gives \v{C}ech-de Rham (double) complex $\check{C}^\bullet(\mathfrak{U},\Omega_{(\mathbf{P}^n\setminus X_{\underline{G}})/\Bbbk}^\bullet)$ and a spectral sequence
\begin{align}\label{XG-complement-ss}
E_2^{p,q}=\check{H}^q\left(\mathfrak{U},\underline{H}^p(\Omega_{(\mathbf{P}^n\setminus X_{\underline{G}})/\Bbbk}^\bullet)\right)\implies H_{dR}^{p+q}(\mathbf{P}^n\setminus X_{\underline{G}};\Bbbk)
\end{align}
where $\underline{H}^p(\Omega_{(\mathbf{P}^n\setminus X_{\underline{G}})/\Bbbk}^\bullet)$ is the $p$-th cohomology sheaf of $\Omega_{(\mathbf{P}^n\setminus X_{\underline{G}})/\Bbbk}^\bullet$.
\begin{rem} Since $\Omega_{(\mathbf{P}^n\setminus X_{\underline{G}})/\Bbbk}^\bullet$ is quasicoherent, its cohomology sheaves are quasicoherent. Since each $D_+(G_{\lambda_0}\cdots G_{\lambda_q})$ are affine,
\begin{align*}
H^r\left(D_+(G_{\lambda_0}\cdots G_{\lambda_q}),\underline{H}^p(\Omega_{(\mathbf{P}^n\setminus X_{\underline{G}})/\Bbbk}^\bullet)\right)=0
\end{align*}
for $r>0$ so
\begin{align*}
\check{H}^q\left(\mathfrak{U},\underline{H}^p(\Omega_{(\mathbf{P}^n\setminus X_{\underline{G}})/\Bbbk}^\bullet)\right)\cong H^q\left(\mathbf{P}^n\setminus X_{\underline{G}},\underline{H}^p(\Omega_{(\mathbf{P}^n\setminus X_{\underline{G}})/\Bbbk}^\bullet)\right)
\end{align*}
Hence the spectral sequence \eqref{XG-complement-ss} represents the one coming from the Cartan-Eilenberg resolution and the derived global section $\mathbf{R}\Gamma(\mathbf{P}^n\setminus X_{\underline{G}},-)$.
\end{rem}
By \eqref{DG-complement}, we have
\begin{align*}
\check{C}^{k-1}\left(\mathfrak{U},\Omega_{(\mathbf{P}^n\setminus X_{\underline{G}})/\Bbbk}^\bullet\right)=\Gamma\left(\mathbf{P}^n\setminus D_{\underline{G}},\Omega_{(\mathbf{P}^n\setminus X_{\underline{G}})/\Bbbk}^\bullet\right)
\end{align*}
so the zeroth page of the spectral sequence \eqref{XG-complement-ss} is given as follows:
\begin{align*}
\xymatrix{
\Gamma\left(\mathbf{P}^n\setminus D_{\underline{G}},\Omega_{(\mathbf{P}^n\setminus X_{\underline{G}})/\Bbbk}^0\right) \ar[r] & \Gamma\left(\mathbf{P}^n\setminus D_{\underline{G}},\Omega_{(\mathbf{P}^n\setminus X_{\underline{G}})/\Bbbk}^1\right) \ar[r] & \cdots \ar[r] & \Gamma\left(\mathbf{P}^n\setminus D_{\underline{G}},\Omega_{(\mathbf{P}^n\setminus X_{\underline{G}})/\Bbbk}^n\right) \\
\vdots \ar[u] & \vdots \ar[u] & \reflectbox{$\ddots$} & \vdots \ar[u] \\
\check{C}^1\left(\mathfrak{U},\Omega_{(\mathbf{P}^n\setminus X_{\underline{G}})/\Bbbk}^0\right) \ar[r] \ar[u] & \check{C}^1\left(\mathfrak{U},\Omega_{(\mathbf{P}^n\setminus X_{\underline{G}})/\Bbbk}^1\right) \ar[r] \ar[u] & \cdots \ar[r] & \check{C}^1\left(\mathfrak{U},\Omega_{(\mathbf{P}^n\setminus X_{\underline{G}})/\Bbbk}^n\right) \ar[u] \\
\check{C}^0\left(\mathfrak{U},\Omega_{(\mathbf{P}^n\setminus X_{\underline{G}})/\Bbbk}^0\right) \ar[r] \ar[u] & \check{C}^0\left(\mathfrak{U},\Omega_{(\mathbf{P}^n\setminus X_{\underline{G}})/\Bbbk}^1\right) \ar[r] \ar[u] & \cdots \ar[r] & \check{C}^0\left(\mathfrak{U},\Omega_{(\mathbf{P}^n\setminus X_{\underline{G}})/\Bbbk}^n\right) \ar[u]
}
\end{align*}
Since $\mathbf{P}^n\setminus D_{\underline{G}}=D_+(G_1\cdots G_k)$ is affine, the top row computes the algebraic de Rham cohomology of $\mathbf{P}^n\setminus D_{\underline{G}}$ over $\Bbbk$ so the next page with respect to the horizontal differential is the following:
\begin{align*}
\xymatrix{
H_{dR}^0(\mathbf{P}^n\setminus D_{\underline{G}};\Bbbk) & H_{dR}^1(\mathbf{P}^n\setminus D_{\underline{G}};\Bbbk) & \cdots & H_{dR}^n(\mathbf{P}^n\setminus D_{\underline{G}};\Bbbk) \\
\vdots \ar[u] & \vdots \ar[u] & \reflectbox{$\ddots$} & \vdots \ar[u] \\
\check{C}^1\left(\mathfrak{U},\underline{H}^0(\Omega_{(\mathbf{P}^n\setminus X_{\underline{G}})/\Bbbk}^\bullet)\right) \ar[u] & \check{C}^1\left(\mathfrak{U},\underline{H}^1(\Omega_{(\mathbf{P}^n\setminus X_{\underline{G}})/\Bbbk}^\bullet)\right) \ar[u] & \cdots & \check{C}^1\left(\mathfrak{U},\underline{H}^n(\Omega_{(\mathbf{P}^n\setminus X_{\underline{G}})/\Bbbk}^\bullet)\right) \ar[u] \\
\check{C}^0\left(\mathfrak{U},\underline{H}^0(\Omega_{(\mathbf{P}^n\setminus X_{\underline{G}})/\Bbbk}^\bullet)\right) \ar[u] & \check{C}^0\left(\mathfrak{U},\underline{H}^1(\Omega_{(\mathbf{P}^n\setminus X_{\underline{G}})/\Bbbk}^\bullet)\right) \ar[u] & \cdots & \check{C}^0\left(\mathfrak{U},\underline{H}^n(\Omega_{(\mathbf{P}^n\setminus X_{\underline{G}})/\Bbbk}^\bullet)\right) \ar[u]
}
\end{align*}
Now the convergence of the spectral sequence (\ref{XG-complement-ss}) gives an exact sequence
\begin{align}\label{XG-complement-ss-es}
\xymatrixcolsep{1.5pc}\xymatrix{0 \ar[r] & \check{C}^{k-2}\left(\mathfrak{U},\underline{H}^n(\Omega_{(\mathbf{P}^n\setminus X_{\underline{G}})/\Bbbk}^\bullet\right) \ar[r] & H_{dR}^n(\mathbf{P}^n\setminus D_{\underline{G}};\Bbbk) \ar[r] & H_{dR}^{n+k-1}(\mathbf{P}^n\setminus X_{\underline{G}};\Bbbk) \ar[r] & 0}
\end{align}

\begin{defn} \label{delgk}
Define $\delta_{\underline{G}}$ to be the surjection in the exact sequence \eqref{XG-complement-ss-es}.
\end{defn}

\subsection{Computation of $\delta_{\protect\underline{G}}$ and a comparison result} \label{sec3.4}

Since $\CMcal{E}=\CMcal{O}_{\BP^n}(d_1) \oplus \cdots \oplus \CMcal{O}_{\BP^n}(d_k)$ is trivialized over $D_+(G_\lambda)$ using $G_\lambda^{-1}$, there is a commutative diagram
\begin{align*}
\xymatrix{
(\mathbf{P}^n\setminus D_{\underline{G}})\times_\Bbbk\mathbf{A}^{k-1} \ar[d] \ar[r] & D_+(G_\lambda)\times_\Bbbk\mathbf{A}^{k-1} \ar[d] \ar[r] & \mathbf{P}(\CMcal{E})\setminus X_S \ar[d]_{pr_1} \\
\mathbf{P}^n\setminus D_{\underline{G}} \ar@{^{(}->}[r] & D_+(G_\lambda) \ar@{^{(}->}[r] & \mathbf{P}^n\setminus X_{\underline{G}}
}
\end{align*}
where the squares are Cartesian. Let
\begin{align*}
\mathfrak{U}\times_\Bbbk\mathbf{A}^{k-1}=\{D_+(G_\lambda)\times_\Bbbk\mathbf{A}^{k-1}\}_{\lambda=1,\cdots,k}
\end{align*}
be the affine open covering of $\mathbf{P}(\CMcal{E})\setminus X_S$ corresponding to (\ref{XG-complement-covering}). Then
\begin{align*}
\check{C}^q\left(\mathfrak{U}\times_\Bbbk\mathbf{A}^{k-1},\underline{H}^p(\Omega_{(\mathbf{P}(\CMcal{E})\setminus X_S)/\Bbbk}^\bullet)\right)\cong\check{C}^q\left(\mathfrak{U},\underline{H}^p(\Omega_{(\mathbf{P}^n\setminus X_{\underline{G}})/\Bbbk}^\bullet)\right)
\end{align*}
for every $(p,q)\in\mathbb{Z}\times\mathbb{Z}$ by the K\"unneth formula and $H_{dR}^\bullet(\mathbf{A}^{k-1})\cong\Bbbk[0]$. 
By the naturality of the spectral sequence associated to the \v{C}ech-de Rham complexes, we get a commutative square
\begin{align}\label{cdl}
\begin{aligned}
\xymatrix{
H_{dR}^n((\mathbf{P}^n\setminus D_{\underline{G}})\times_\Bbbk\mathbf{A}^{k-1};\Bbbk) \ar[r]_-{\delta_D} & H_{dR}^{n+k-1}(\mathbf{P}(\CMcal{E})\setminus X_S;\Bbbk) \\
H_{dR}^n(\mathbf{P}^n\setminus D_{\underline{G}};\Bbbk) \ar[u]^\wr_-{pr_1^*}\ar[r]_-{\delta_{\underline{G}}} & H_{dR}^{n+k-1}(\mathbf{P}^n\setminus X_{\underline{G}};\Bbbk) \ar[u]^-\wr_-{pr_1^*}
}
\end{aligned}
\end{align}
where the column isomorphism come from the projection $pr_1:\mathbf{P}(\CMcal{E})\setminus X_S\rightarrow\mathbf{P}^n\setminus X_{\underline{G}}$ and the top row $\delta_D$ is defined from an exact sequence analogous to (\ref{XG-complement-ss-es}). 

%

Now we will represent $\delta_D$ and $\delta_{\underline{G}}$ analogously as in \cite[section 3]{Dim95} or \cite[pp. 651--654]{GH}, but using only polynomial differential forms. For this, we choose a map of affine schemes
\begin{align}\label{DG-complement-section}
\xymatrix{\sigma_{\underline{G}}:\mathbf{P}^n\setminus D_{\underline{G}} \ar[r] & (\mathbf{P}^n\setminus D_{\underline{G}})\times_\Bbbk\mathbf{A}^{k-1}}
\end{align}
given by the ring homomorphism
\begin{align*}
\xymatrix{\displaystyle\Bbbk\left[\underline{x},\underline{y},\frac{1}{G_1\cdots G_k},\frac{1}{S}\right]_{0,(0)} \ar[r] & \displaystyle\Bbbk\left[\underline{x},\frac{1}{G_1\cdots,G_k}\right]_0 & (\underline{x},y_1,\cdots,y_k) \ar@{|->}[r] & \displaystyle\left(\underline{x},\frac{1}{G_1},\cdots,\frac{1}{G_k}\right)}
\end{align*}
which is a section to the projection $(\mathbf{P}^n\setminus D_{\underline{G}})\times_\Bbbk\mathbf{A}^{k-1}\rightarrow\mathbf{P}^n\setminus D_{\underline{G}}$ (recall that sub-indices $0$ and $(0)$ mean the charge 0 and weight 0 submodule respectively). Using the presentation
\begin{align*}
(\mathbf{P}^n\setminus D_{\underline{G}})\times_\Bbbk\mathbf{A}^{k-1}\cong\Spec\left( \Bbbk\left[\underline{x},\frac{1}{G_1\cdots G_k}\right]_{0}\left[\frac{y_1G_1}{S},\cdots,\frac{y_kG_k}{S}\right]\right)
\end{align*}
we see that the above map can be written as
\begin{align*}
\xymatrix{\mathbf{P}^n\setminus D_{\underline{G}} \ar[r] & (\mathbf{P}^n\setminus D_{\underline{G}})\times_\Bbbk\mathbf{A}^{k-1} & \underline{x} \ar@{|->}[r] & \displaystyle\left(\underline{x},\frac{1}{k},\cdots,\frac{1}{k}\right)}.
\end{align*}
This suggests a ``partition of unity''
\begin{align*}
\left\{\frac{S_\lambda}{S}\ \middle|\ S_\lambda:=y_\lambda G_\lambda,\quad\lambda=1,\cdots,k\right\}
\end{align*}
Here the quotation mark means that this ``partition of unity'' works only for the differential forms with poles along the divisor $D_{\underline{G}}\times_\Bbbk\mathbf{A}^{k-1}$ of order at most $1$. Any differential $n$-form on $\mathbf{P}^n\setminus D_{\underline{G}}$ can be written as
\begin{align*}
\frac{1}{G_1^{i_1+1}\cdots G_k^{i_k+1}}\xi,\quad\xi\in\Gamma\left(\mathbf{P}^n\setminus X_{\underline{G}},\Omega_{(\mathbf{P}^n\setminus X_{\underline{G}})/\Bbbk}^n\right)
\end{align*}
such that $\ch(\xi)=\ch(G_1^{i_1+1}\cdots G_k^{i_k+1})$. But this can be lifted via $\sigma_{\ud G}$ to a differential form $\tilde \xi$ with poles along the divisor $D_{\underline{G}}\times_\Bbbk\mathbf{A}^{k-1}$ of order at most $1$:
\begin{align}\label{sigmag}
\sigma_{\ud G}^* \left(\tilde \xi \right) = \frac{1}{G_1^{i_1+1}\cdots G_k^{i_k+1}}\xi, \quad \quad \tilde \xi:=\frac{1}{G_1\cdots G_k}\frac{k^{|\underline{i}|}\underline{y}^{\underline{i}}}{S^{|\underline{i}|}}\xi\in\G\left((\BP^n\setminus D_{\ud G}) \times_\Bbbk\mathbf{A}^{k-1},\Omega_{(\mathbf{P}(\CMcal{E})\setminus X_S)/\Bbbk}^n\right)
\end{align}
where $\underline{i}=(i_1,\cdots,i_k)$ denotes a multi-index.


For $\omega\in\check{C}^q\left(\mathfrak{U}\times_\Bbbk\mathbf{A}^{k-1},\Omega_{(\mathbf{P}(\CMcal{E})\setminus X_S)/\Bbbk}^p\right)$ with poles along the divisor $D_{\underline{G}}\times_\Bbbk\mathbf{A}^{k-1}$ of order at most $1$, we define
\begin{align*}
(\tau\omega)_{\lambda_0\cdots\lambda_{q-1}}:=\sum_{\lambda=1}^k\frac{S_\lambda}{S}\omega_{\lambda\lambda_0\cdots\lambda_{q-1}}\in\Gamma\left(D_+(G_{\lambda_0}\cdots G_{\lambda_{q-1}})\times_\Bbbk\mathbf{A}^{k-1},\Omega_{(\mathbf{P}(\CMcal{E})\setminus X_S)/\Bbbk}^p\right).
\end{align*}
 Note that terms with repeated indexes are regarded as zero. Hence we get an element
\begin{align*}
\tau\omega\in\check{C}^{q-1}\left(\mathfrak{U}\times_\Bbbk\mathbf{A}^{k-1},\Omega_{(\mathbf{P}(\CMcal{E})\setminus X_S)/\Bbbk}^p\right).
\end{align*}

Let 
\begin{align*}
\xymatrix{d_\uparrow:\check{C}^q\left(\mathfrak{U}\times_\Bbbk\mathbf{A}^{k-1},\Omega_{(\mathbf{P}(\CMcal{E})\setminus X_S)/\Bbbk}^p\right) \ar[r] & \check{C}^{q+1}\left(\mathfrak{U}\times_\Bbbk\mathbf{A}^{k-1},\Omega_{(\mathbf{P}(\CMcal{E})\setminus X_S)/\Bbbk}^p\right)}
\end{align*}
be the vertical differential coming from the \v{C}ech differential;
let
\begin{align*}
\xymatrix{d_\rightarrow:\check{C}^q\left(\mathfrak{U}\times_\Bbbk\mathbf{A}^{k-1},\Omega_{(\mathbf{P}(\CMcal{E})\setminus X_S)/\Bbbk}^p\right) \ar[r] & \check{C}^q\left(\mathfrak{U}\times_\Bbbk\mathbf{A}^{k-1},\Omega_{(\mathbf{P}(\CMcal{E})\setminus X_S)/\Bbbk}^{p+1}\right)}
\end{align*}
be the horizontal differential coming from the de Rham differential.

\begin{lem}\label{clem}
For
\begin{align*}
\omega=\omega_{1\cdots k}\in\check{C}^{k-1}\left(\mathfrak{U}\times_\Bbbk\mathbf{A}^{k-1},\Omega_{(\mathbf{P}(\CMcal{E})\setminus X_S)/\Bbbk}^n\right) 
\end{align*}
such that $d_\rightarrow\omega=0$ ($d_\uparrow\omega=0$ is automatic),
\begin{align*}
(-d_\rightarrow\tau)^{k-1}\omega\in\check{C}^0\left(\mathfrak{U}\times_\Bbbk\mathbf{A}^{k-1},\Omega_{(\mathbf{P}(\CMcal{E})\setminus X_S)/\Bbbk}^{n+k-1}\right) 
\end{align*}
represents the cohomology class $\delta_{D}([\omega])$.
\end{lem}

\begin{proof}
For $\omega\in\check{C}^q\left(\mathfrak{U}\times_\Bbbk\mathbf{A}^{k-1},\Omega_{(\mathbf{P}(\CMcal{E})\setminus X_S)/\Bbbk}^p\right)$ with poles along the divisor $D_{\underline{G}}\times_\Bbbk\mathbf{A}^{k-1}$ of order at most $1$, we compute
\begin{align*}
(d_\uparrow\tau\omega)_{\lambda_0\cdots\lambda_q}=\sum_{j=0}^q(-1)^j(\tau\omega)_{\lambda_0\cdots\widehat{\lambda}_j\cdots\lambda_q}=\sum_{j=0}^q\sum_{\lambda=1}^k(-1)^j\frac{S_\lambda}{S}\omega_{\lambda\lambda_0\cdots\widehat{\lambda}_j\cdots\lambda_q}.
\end{align*}
This implies that
\begin{align*}
(\tau d_\uparrow\omega)_{\lambda_0\cdots\lambda_q}&=\sum_{\lambda=1}^k\frac{S_\lambda}{S}(d_\uparrow\omega)_{\lambda\lambda_0\cdots\lambda_q}\\
&=\sum_{\lambda=1}^k\left(\frac{S_\lambda}{S}\omega_{\lambda_0\cdots\lambda_q}+\sum_{j=0}^q(-1)^{j+1}\frac{S_\lambda}{S}\omega_{\lambda\lambda_0\cdots\widehat{\lambda}_j\cdots\lambda_q}\right)\\
&=\omega_{\lambda_0\cdots\lambda_q}-\sum_{j=0}^q\sum_{\lambda=1}^k(-1)^j\frac{S_\lambda}{S}\omega_{\lambda\lambda_0\cdots\widehat{\lambda}_j\cdots\lambda_q}\\
&=\omega_{\lambda_0\cdots\lambda_q}-(d_\uparrow\tau\omega)_{\lambda_0\cdots\lambda_q}
\end{align*}
i.e., $(d_\uparrow\tau+\tau d_\uparrow)\omega=\omega$ (cf. \cite[Proposition 8.5]{BT}). 

 Now assume that $d_\uparrow\omega=0$ and $d_\rightarrow\omega=0$. In this case, $\omega_1:=-d_\rightarrow\tau\omega$ satisfies
\begin{align*}
d_\rightarrow\omega_1=-d_\rightarrow^2\tau\omega=0,\quad d_\uparrow\omega_1=-d_\uparrow d_\rightarrow\tau\omega=d_\rightarrow d_\uparrow\tau\omega=d_\rightarrow(\omega-\tau d_\uparrow\omega)=0
\end{align*}
Hence we may apply $-d_\rightarrow\tau$ to $\omega_1$ again. Moreover,
\begin{align*}
(d_\uparrow+d_\rightarrow)\tau\omega=d_\uparrow\tau\omega-\omega_1=\omega-\tau d_\uparrow\omega-\omega_1=\omega-\omega_1
\end{align*}
so $\omega$ and $\omega_1=-d_\rightarrow\tau\omega$ define the same class in the cohomology of the total complex. In particular, for 
$
\omega=\omega_{1\cdots k}\in\check{C}^{k-1}\left(\mathfrak{U}\times_\Bbbk\mathbf{A}^{k-1},\Omega_{(\mathbf{P}(\CMcal{E})\setminus X_S)/\Bbbk}^n\right) 
$
such that $d_\rightarrow\omega=0$ ($d_\uparrow\omega=0$ is automatic),
$
(-d_\rightarrow\tau)^{k-1}\omega\in\check{C}^0\left(\mathfrak{U}\times_\Bbbk\mathbf{A}^{k-1},\Omega_{(\mathbf{P}(\CMcal{E})\setminus X_S)/\Bbbk}^{n+k-1}\right) 
$
represents the cohomology class $\delta_{D}([\omega])$ by the definition of $\delta_{D}$.
\end{proof}



Recall the maps $k_{\ud G}:  A_{c_{\ud G}}/A_{c_{\ud G}} \cap Jac(S) \to H^{n+k-1}_{dR}(\BP(\CMcal{E}) \setminus X_S; \Bbbk)$ and $d_{\ud G}:  A_{c_{\ud G}}/A_{c_{\ud G}} \cap Jac(S) \to H^n_{dR}(\BP^n \setminus D_{\ud G};\Bbbk) $ from \eqref{kg} and \eqref{dg}:
  \begin{eqnarray*} 
  k_{\ud G} ([\ud y^{\ud i}u(\ud x)]) &:=& [\a(S,\ud y^{\ud i}u(\ud x) )]=\left[(-1)^{n(k-1)+|\ud i|}(k+|\ud i|-1)!\frac{\ud y^{\ud i}u(\ud x) }{S(\ud q)^{k+|\ud i|}} \Omega_x \Omega_y\right]\\
  d_{\ud G} ([\ud y^{\ud i}u(\ud x)]) &:=& [\b(S,\ud y^{\ud i}u(\ud x) )]=\left[(-1)^{|\ud i|+1}\frac{i_1! \cdots i_k!\cdot u(\ud x)}{G_1^{i_1+1}(\ud x) \cdots G_k^{i_k+1}(\ud x)} \Omega_x \right].
  \end{eqnarray*}
  
%

\begin{thm}\label{mthm}
We have
\begin{align*}
(pr_1^\ast\circ \delta_{\underline{G}}) 
\left[\b(S,\ud y^{\ud i} u(\ud x))] \right)
=\frac{i_1! \cdots i_k! (k-1)!k^{|\ud i|}}{(-1)^{|\ud i|+1}(|\ud i|+k-1)!}\left[\a(S,\ud y^{\ud i} u(\ud x))] \right)
\end{align*}
for $\underline{y}^{\underline{i}}u(\underline{x}) \in A_{c_{\ud G}}.$ In other words, we have
$pr_1^* \circ \delta_{\ud G} \circ d_{\ud G} = \frac{i_1! \cdots i_k! (k-1)!k^{|\ud i|}}{(-1)^{|\ud i|+1}(|\ud i|+k-1)!}\cdot k_{\ud G}.$
\end{thm}
\begin{proof}
By the commutativity $pr_1^\ast \circ \delta_{\ud G}\circ (pr_1^\ast)^{-1}=\delta_D$ of the diagram \eqref{cdl} and the fact $(pr_1^*)^{-1}=\sigma_{\ud G}^*$, it suffices to construct a differential form
$
\omega(S,\underline{y}^{\underline{i}}u(\underline{x}))\in\Gamma\left((\mathbf{P}^n\setminus D_{\underline{G}})\times_\Bbbk\mathbb{A}^{k-1},\Omega_{(\mathbf{P}(\CMcal{E})\setminus X_S)/\Bbbk}^n\right)
$
which satisfies
\begin{align}\label{fresult}
\sigma_{\ud G}^\ast ([\omega(S,\underline{y}^{\underline{i}}u(\underline{x})) ])=\frac{(-1)^k(|\ud i|+k-1)!}{i_1! \cdots i_k! (k-1)!k^{|\ud i|}}[\b(S,\ud y^{\ud i}u(\ud x) )], \quad
\delta_D ([\omega(S,\underline{y}^{\underline{i}}u(\underline{x})) ])=[\a(S,\ud y^{\ud i}u(\ud x) )].
\end{align}

We claim that the following differential form
$$
\omega(S,\underline{y}^{\underline{i}}u(\underline{x})):=(-1)^{|\underline{i}|+k-1}\frac{(|\underline{i}|+k-1)!}{(k-1)!}\frac{1}{G_1\cdots G_k}\frac{\underline{y}^{\underline{i}}u(\underline{x})}{S^{|\underline{i}|}}\Omega_x
\in\Gamma\left((\mathbf{P}^n\setminus D_{\underline{G}})\times_\Bbbk\mathbb{A}^{k-1},\Omega_{(\mathbf{P}(\CMcal{E})\setminus X_S)/\Bbbk}^n\right)
$$
serves our purpose. The first equality in \eqref{fresult} clearly holds from the definition \eqref{sigmag} of $\sigma_{\ud G}^*$.
For the second equality, based on Lemma \ref{clem}, we compute $(-d_\rightarrow\tau)^{k-1}\omega$ on $D_+(G_\lambda)\times_\Bbbk\mathbf{A}^{k-1}$: its component is
\begin{align*}
\left((-d_\rightarrow\tau)^{k-1}\omega\right)_\lambda&=(-1)^{k-1}\left(\sum_{\mu_1=1}^kd\left(\frac{S_{\mu_1}}{S}\right)\right)\wedge\cdots\wedge\left(\sum_{\mu_{k-1}=1}^kd\left(\frac{S_{\mu_{k-1}}}{S}\right)\right)\wedge\omega_{\mu_1\cdots\mu_{k-1}\lambda}\\
&=(-1)^{k-1}(k-1)!d\left(\frac{S_1}{S}\right)\wedge\cdots\wedge\widehat{d\left(\frac{S_\lambda}{S}\right)}\wedge\cdots\wedge d\left(\frac{S_k}{S}\right)\wedge\omega_{1\cdots k}.
\end{align*}
From the computation
\begin{align*}
&\quad d\left(\frac{S_1}{S}\right)\wedge\cdots\widehat{d\left(\frac{S_\lambda}{S}\right)}\wedge\cdots\wedge d\left(\frac{S_k}{S}\right)\\
&=\frac{dS_1\wedge\cdots\wedge\widehat{dS_\lambda}\wedge\cdots\wedge dS_k}{S^{k-1}}-\sum_{\mu\neq\lambda}\frac{dS_1\wedge\cdots\wedge\widehat{dS_\lambda}\wedge\cdots\wedge S_\mu dS\wedge\cdots\wedge dS_k}{S^k}\\
&=\frac{S-(S_1+\cdots+\widehat{S}_\lambda+\cdots+S_k)}{S^k}dS_1\wedge\cdots\wedge\widehat{dS_\lambda}\wedge\cdots\wedge dS_k\\
&\quad\quad+\sum_{\mu\neq\lambda}(-1)^{\lambda-\mu}\frac{S_\mu dS_1\wedge\cdots\wedge dS_\lambda\wedge\cdots\wedge\widehat{dS_\mu}\wedge\cdots\wedge dS_k}{S^k}\\
&=(-1)^\lambda\sum_{\mu=1}^k(-1)^\mu\frac{S_\mu dS_1\wedge\cdots\wedge\widehat{dS_\mu}\wedge\cdots\wedge dS_k}{S^k}
\end{align*}
we obtain
\begin{align*}
\left((-d_\rightarrow\tau)^{k-1}\omega\right)_\lambda=(-1)^\lambda(-1)^{k-1}(k-1)!\sum_{\mu=1}^k(-1)^\mu\frac{S_\mu dS_1\wedge\cdots\wedge\widehat{dS_\mu}\wedge\cdots\wedge dS_k}{S^k}\wedge\omega_{1\cdots k}
\end{align*}
Since the local sections agree up to $(-1)^\lambda$, which is the index of $D_+(G_\lambda)$, they glue to define a global section of $\Omega_{(\mathbf{P}(\CMcal{E})\setminus X_S)/\Bbbk}^{n+k-1}$:
\begin{align*}
\delta(\omega)=(-1)^{k-1}(k-1)!\sum_{\lambda=1}^k(-1)^\lambda\frac{S_\mu dS_1\wedge\cdots\wedge\widehat{dS_\lambda}\wedge\cdots\wedge dS_k}{S^k}\wedge\omega
\end{align*}
i.e. if we extend the \v{C}ech complex as usual:
\begin{align*}
\xymatrix{d_\uparrow:\Gamma\left(\mathbf{P}(\CMcal{E})\setminus X_S,\Omega_{(\mathbf{P}(\CMcal{E})\setminus X_S)/\Bbbk}^{n+k-1}\right) \ar[r] & \check{C}^0\left(\mathfrak{U}\times_\Bbbk\mathbf{A}^{k-1},\Omega_{(\mathbf{P}(\CMcal{E})\setminus X_S)/\Bbbk}^{n+k-1}\right) & \psi \ar@{|->}[r] & \left((-1)^\lambda\psi\right)_{\lambda=1,\cdots,k}}
\end{align*}
then $d_\uparrow\delta(\omega)=(-d_\rightarrow\tau)^{k-1}\omega$. By the sheaf property, $d_\uparrow d_\rightarrow\delta(\omega)=-d_\rightarrow d_\uparrow\delta(\omega)=0$ implies that $\delta(\omega)$ is also a de Rham cocycle. Thus $\delta_D([\omega])=[\delta(\omega)]$. 

Note that
\begin{align*}
\delta(\Omega_x)&=(-1)^{k-1}(k-1)!\sum_{\lambda=1}^k(-1)^\lambda\frac{S_\lambda dS_1\wedge\cdots\wedge\widehat{dS_\lambda}\wedge\cdots\wedge dS_k}{S^k}\wedge\Omega_x\\
&\equiv(-1)^{(k-1)(n-1)}(k-1)!\frac{G_1\cdots G_k}{S^k}\Omega_x\wedge\Omega_y\quad \mod dx_0\wedge\cdots\wedge dx_n
\end{align*}
together with the vanishing of $[dx_0\wedge\cdots\wedge dx_n]$ in the cohomology shows that 
\begin{align*}
[\delta(\Omega_x)]=\left[(-1)^{(k-1)(n-1)}(k-1)!\frac{G_1\cdots G_k}{S^k}\Omega_x\wedge\Omega_y\right].
\end{align*}
Therefore, we get
\begin{align*}
\delta\left({\omega(S,\underline{y}^{\underline{i}}u(\underline{x}))}\right)\equiv(-1)^{|\underline{i}|+n(k-1)}(|\underline{i}|+k-1)!\frac{\underline{y}^{\underline{i}}u(\underline{x})}{S^{|\underline{i}|+k}}\Omega_x\wedge\Omega_y \quad \mod dx_0\wedge\cdots\wedge dx_n,
\end{align*}
which proves the second equality in \eqref{fresult}.
\end{proof}

\end{document}